\documentclass[11pt]{article}
\usepackage{fullpage}

\usepackage{amssymb,amsmath,amsthm}
\usepackage{epsfig}
\usepackage{url}

\makeatletter
\let\@@citation@@=\citation
\renewcommand{\citation}[1]{\@@citation@@{#1}%
\@for\@tempa:=#1\do{\@ifundefined{cit@\@tempa}%
  {\global\@namedef{cit@\@tempa}{}}{}}%
}
\makeatother

\usepackage{hyperref}
\makeatletter
\def\@lbibitem[#1]#2#3\par{%
  \@ifundefined{cit@#2}{}{\@skiphyperreftrue
  \H@item[%
    \ifx\Hy@raisedlink\@empty
      \hyper@anchorstart{cite.#2\@extra@b@citeb}%
        \@BIBLABEL{#1}%
      \hyper@anchorend
    \else
      \Hy@raisedlink{%
        \hyper@anchorstart{cite.#2\@extra@b@citeb}\hyper@anchorend
      }%
      \@BIBLABEL{#1}%
    \fi
    \hfill
  ]%
  \@skiphyperreffalse}%
  \if@filesw
    \begingroup
      \let\protect\noexpand
      \immediate\write\@auxout{%
        \string\bibcite{#2}{#1}%
      }%
    \endgroup
  \fi
  \ignorespaces
  \@ifundefined{cit@#2}{}{#3}}

\def\@bibitem#1#2\par{%
  \@ifundefined{cit@#1}{}{\@skiphyperreftrue\H@item\@skiphyperreffalse
  \Hy@raisedlink{%
    \hyper@anchorstart{cite.#1\@extra@b@citeb}\relax\hyper@anchorend
    }}%
  \if@filesw
    \begingroup
      \let\protect\noexpand
      \immediate\write\@auxout{%
        \string\bibcite{#1}{\the\value{\@listctr}}%
      }%
    \endgroup
  \fi
  \ignorespaces
  \@ifundefined{cit@#1}{}{#2}}
\makeatother

\makeatletter
\newcommand{\linkhere}[2]{%
	\phantomsection
	#1\def\@currentlabel{\unexpanded{#1}}\label{#2}%
}
\makeatother

\newtheorem{thm}{Theorem}%[section]
\newtheorem{cor}[thm]{Corollary}
\newtheorem{lem}[thm]{Lemma}

\newtheorem{conj}[thm]{Conjecture}

\newtheorem{obs}[thm]{Observation}

\theoremstyle{definition}

\usepackage{indentfirst}
\usepackage{xspace} % A makrok utani space korrekt kezelese miatt.

\def\F{\mbox{\ensuremath{\mathcal F}}\xspace}
\def\HH{\mbox{\ensuremath{\mathcal H}}\xspace}

\mathcode`l="8000
\begingroup
\makeatletter
\lccode`\~=`\l
\DeclareMathSymbol{\lsb@l}{\mathalpha}{letters}{`l}
\lowercase{\gdef~{\ifnum\the\mathgroup=\m@ne \ell \else \lsb@l \fi}}%
\endgroup

\begin{document}

\title{Note on polychromatic coloring of hereditary hypergraph families}
\author{D\"om\"ot\"or P\'alv\"olgyi
\footnote{ELTE E\"otv\"os Lor\'and University and Alfr\'ed R\'enyi Institute of Mathematics, Budapest, Hungary.\\
Supported by the ERC Advanced Grant ``ERMiD'' and by the J\'anos Bolyai Research Scholarship of the Hungarian Academy of Sciences, and by the New National Excellence Program \'UNKP-22-5 and by the Thematic Excellence Program TKP2021-NKTA-62 of the National Research, Development and Innovation Office.}}

%külön:
%Research supported by the Lend\"ulet program of the Hungarian Academy of Sciences (MTA), under grant number LP2017-19/2017
% and by the J\'anos Bolyai Research Scholarship of the Hungarian Academy of Sciences
%and by the New National Excellence Program of the Ministry for Innovation and Technology from the source of the National Research, Development and Innovation Fund, under grants \'UNKP-22-5
%and by the Ministry of Innovation and Technology of Hungary from the National Research, Development and Innovation Fund, financed under the  ELTE TKP 2021-NKTA-62 funding scheme.

%\date{}
\maketitle

\begin{abstract}
We exhibit a 5-uniform hypergraph that has no polychromatic 3-coloring, but all its restricted subhypergraphs with edges of size at least 3 are 2-colorable.
This disproves a bold conjecture of Keszegh and the author, and can be considered as the first step to understand polychromatic colorings of hereditary hypergraph families better since the seminal work of Berge.
We also show that our method cannot give hypergraphs of arbitrary high uniformity, and mention some connections to panchromatic colorings.
\end{abstract}

\medskip

A hypergraph $\HH=(V,E)$ is a collection of sets $E$ over a base set $V$.
The elements of $V$ are called the vertices of the hypergraph and the elements of $E$ the edges of the hypergraph.
%A hypergraph is finite if $V$ and $E$ are finite sets.
A hypergraph is $m$-uniform if all of its edges have size $m$ and it is $m$-heavy if all of its edges have size at least $m$.
%Edges with at least $m$ vertices are called $m$-heavy; the $m$ is omitted when it leads to no confusion.

$\HH'=(V',E')$ is a subhypergraph of $\HH=(V,E)$ if $V'\subset V$ and $E'\subset E$.
For a subset of the vertices $X\subset V$, we define the trace of \HH on $X$ as the hypergraph $\HH[X]=(X,E\cap X)$, where $E\cap X=\{e\cap X\mid e\in E(\HH)\}$; we also call this the restriction of \HH to $X$.
The subhypergraph of a trace is called a restricted subhypergraph.
A family of hypergraphs \F is hereditary if it is closed for taking subhypergraphs and traces, i.e., all the restricted subhypergraphs of any $\HH\in\F$ are also in \F.
From any hypergraph \HH we can get a hereditary \F by taking the family of all restricted subhypergraphs of \HH.

A $k$-coloring of the vertices of a hypergraph \HH is polychromatic (aka.\ panchromatic) if every edge contains a vertex of each $k$ colors.\footnote{In the special case $k=2$ this coincides with the definition of proper 2-colorings. If a 2-coloring exists, we say that \HH is 2-colorable, or that \HH has Property B.}
If such a coloring exists, then \HH needs to be $k$-heavy (but this is not a sufficient condition).
Because of this, define $m_k(\F)$ to be the least $m$ (if exists) for which every $m$-heavy $\HH\in \F$ is polychromatic $k$-colorable; if there is no such $m$, then let $m_k(\F)=\infty$.
If every edge of every hypergraph of a family is smaller than $m$, then $m_k(\F)\le m$ by definition.
Note that by definition we always have $m_k(\F)\le m_{k+1}(\F)$.
Moreover, the following might also be true.

\begin{conj}\label{conj:mk}
	If $m_2(\F)<\infty$, then $m_k(\F)<\infty$ for every $k$ for any hereditary family $\F$.
\end{conj}

This beautiful conjecture is only known to hold if $m_2(\F)=2$; this is a classic result of Berge \cite{Ber72}, who proved that in this case $m_k(\F)=k$, and also characterized these families.

Conjecture \ref{conj:mk} first ``came up'' in 2009 during the writing of a survey \cite{survey}, until then it was simply believed to be true by the (few) people working on related questions for geometric families; for some of the latest results in this area, see \cite{anchored,BCsSz23,narmada,chekan,DP20,DP22,KP14,KPpropcol,PU23}, or the webpage \url{https://coge.elte.hu/cogezoo.html}, maintained by Keszegh and the author.

Later, Conjecture \ref{conj:mk} was popularized by the author at several venues, such as Oberwolfach meetings and MathOverflow, and other variants also emerged.
The strongest form is the following, conjectured by Keszegh and the author, inspired by geometric families for which the below inequality is sharp.

\begin{conj}\label{conj:mklin}
	$m_k(\F)\le (k-1)(m_2(\F)-1)+1$ for every $k$ for any hereditary family $\F$.
\end{conj}

The goal of this note is to disprove Conjecture \ref{conj:mklin} in the following very weak form.

\begin{thm}\label{thm:main}
	There exists a $5$-uniform hypergraph \HH that has no polychromatic $3$-coloring, but all the $3$-heavy restricted subhypergraphs of \HH are $2$-colorable.
	Therefore, for the hereditary family \F of restricted subhypergraphs of \HH we have $m_3(\F)=6$ and $m_2(\F)=3$.
\end{thm}

The proof is based on the following simple observation, which is the equivalent of the well-known $\frac{\alpha(G)}{|V(G)|}\le \frac1{\chi(G)}$ bound for graphs, stated below for hypergraphs and polychromatic colorings.

\begin{obs}\label{obs:alpha}
	If $\frac{\alpha(\HH)}{|V(\HH)|}<\frac{k-1}k$, then \HH has no polychromatic $k$-coloring.
\end{obs}
\begin{proof}
	If \HH has a polychromatic $k$-coloring, then by the pigeonhole principle one of the color classes, $S\subset V=V(\HH)$, has size at most $\lfloor\frac{|V|}k\rfloor$.
	Since every edge must intersect $S$ because of the polychromacity of the coloring, $V\setminus S$ is an independent set, so $\alpha(\HH)\ge |V|-\lfloor\frac{|V|}k\rfloor=\lceil\frac{k-1}k|V|\rceil\ge \frac{k-1}k|V|$.
\end{proof}	

\begin{proof}[Proof of Theorem \ref{thm:main}.]
	Let $V=\{1,\dots,8\}$,
	$E=\{\{4,5,6,7,8\},\{2,3,5,6,8\},\{2,3,4,7,8\},\{1,3,5,7,8\}$,\\	
	$\{1,3,4,6,8\},\{1,2,6,7,8\},\{1,2,4,5,8\},\{3,4,5,6,7\},\{1,2,4,5,7\},\{1,2,3,5,6\},\{2,3,4,6,7\}\}$.
	
	Note that the complements of the first 7 edges give the Fano plane over $\{1,\dots,7\}$.
	For every $1\le i\le 7$, the complement of one of the remaining 4 edges contains $\{i,8\}$, i.e., the edge will miss $\{i,8\}$.
	Thus, every pair of vertices is missed by at least one edge, i.e., there is an edge in any sextuple of vertices, so $\alpha\le 5<\frac23\cdot 8$.
	Using Observation \ref{obs:alpha}, \HH cannot have a polychromatic 3-coloring.
	In this special case, we could also argue that by the pigeonhole principle one of the color classes would consist of at most 2 vertices, and thus would be missed by an edge.
	
	I do not have an elegant proof for the fact that all $3$-heavy restricted subhypergraphs of this hypergraph are $2$-colorable.
	One way is to check this, is by hand.
	If we restrict to some $X$ with $|X|\le 4$, then coloring $\lfloor\frac{|X|}2\rfloor$ of the vertices blue and $\lceil\frac{|X|}2\rceil$ of the vertices red, will automatically intersect each edge of size at least 3.
	Such a coloring also works for $|X|=5$, unless the restricted subhypergraph contains a $K_5^{(3)}$, a complete 3-uniform hypergraph on 5 vertices.
	However, we only have 11 edges, and no 10 of them form a $K_5^{(3)}$.
	Similarly, for $|X|=6$ there is no 2-coloring only if we have $\frac{\binom63}2=10$ edges that cut $X$ into two equal parts in every possible way, which is not the case.
	For $|X|=7$, we can take one of the $\binom73=35$ triples that is not missed by any of the 11 edges, color them blue, and the other 4 vertices red.
	For $|X|=8$, coloring any 4--4 vertices blue and red works.
	
	I have also verified this with a SAT solver, as the 2-colorability of a hypergraph can be expressed very simply as a SAT formula, and there are only $2^8$ restrictions, most of them trivial.
	In fact, even instead of the SAT solver, a brute force search of 2-colorings would run fast enough.
\end{proof}

\noindent
\textbf{Remarks.}
I have found the 4 extra edges by ad hoc methods.
Some 4-tuples work, some others do not.
It is also possible that we can start with some completely different configuration instead of the Fano plane to obtain 11 edges such that each pair of vertices is missed by at least one edge. 
However, it is easy to see that a 5-uniform hypergraph on 8 vertices with 10 edges cannot be sufficient:
Each edge misses only 3 pairs, and there are $\binom82=28$ pairs to be missed, but because of the parity, every vertex is included in a pair that is missed at least twice, so at least $\frac{28+8/2}3$ 5-edges are required to miss every pair at least once.

It is also easy to prove that there can be no such hypergraph on 7 vertices.
From the fact that there is no polychromatic 3-coloring with 3 green, 2 blue and 2 red vertices, we can conclude that the graph of the pairs of vertices \emph{not} missed by the 5-edges cannot have 2 disjoint 2-edges, as otherwise we could color the vertices of one of these 2-edges blue, the vertices of the other 2-edge red.
In other words, the graph of non-missed pairs must form a star or a triangle.
In each case there need to be many 5-edges, and it is easy to find a $K_5^{(3)}$ restricted subhypergraph, which is 3-heavy and not 2-colorable.

This shows that our example is minimal in the sense that there are no examples on less vertices, or on the same number of vertices, but with less edges.
I could not rule out the possibility of an example with less edges and more vertices.
A strongly related problem is to determine the least number $p(m,k)$ for which there is an
$m$-uniform hypergraph with $p(m,k)$ edges that does not have a polychromatic $k$-coloring.
This problem has a rich history, under the name panchromatic coloring.\\ 

\textbf{Panchromatic colorings.}
%https://arxiv.org/abs/2008.03827
Assume that there is an $m$ and $k$ for which $p((m-1)(k-1)+1,k)<p(m,2)$, and consider an $((m-1)(k-1)+1)$-uniform hypergraph \HH with $p((m-1)(k-1)+1,k)$ edges that has no polychromatic $k$-coloring.
Then in the family \F of all restricted subhypergraphs of \HH, each hypergraph has less than $p(m,2)$ edges, and is thus 2-colorable.
That is, $m_2(\F)\le m$ and $m_k(\F)>(m_2(\F)-1)(k-1)+1$, giving a counterexample to Conjecture \ref{conj:mklin}.

However, I could not find any $k$ and $m$ for which $p((m-1)(k-1)+1,k)<p(m,2)$ would hold.
It is known that $p(M,k)$ is roughly $(k/(k-1))^M$ for large $M$, from which $p((m-1)(k-1)+1,k)\approx (k/(k-1))^{(m-1)(k-1)+1}\approx ((1+\frac1{k-1})^{k-1})^m$.
As $(1+\frac1{k-1})^{k-1}$ is monotone increasing, we cannot get a counterexample from this argument for large $m$, but I could not rule out the possibility for the inequality to hold for some small $m$.
For example, it is known that $p(3,2)=7$, %https://en.wikipedia.org/wiki/Property_B#Known_values_of_m(n)
and I could not find any references whether $p(5,3)\ge 7$ holds or not, though most likely it does.\\

\textbf{Stronger counterexamples.}
It is tempting to try to generalize the pigeonhole principle argument from the proof of Theorem \ref{thm:main} by taking a $t$-uniform hypergraph \HH on $n$ vertices that has no polychromatic 3-coloring because any $\lfloor\frac n3\rfloor$-tuple of vertices is missed by some edge of \HH.
But this should be done very carefully if we want to ensure that all 3-heavy restricted subhypergraphs are 2-colorable.
For example, for each such subhypergraph $\HH'$ on $n'$ vertices $\alpha(\HH')\ge \frac{n'}2$ needs to hold by Observation \ref{obs:alpha}.
The following statements imply that also the VC-dimension of such an \HH can be at most 4.

\begin{lem}\label{lem:VC}
	If $VC\text{--}dim(\HH)\ge kd-1$ for some integer $d$ and $\HH\in \F$, then $m_k(\F)>(k-1)d$.
\end{lem}
\begin{proof}
	If \HH contains a $K_{kd-1}^{((k-1)d)}$, there is a color class that contains at most $d-1$ vertices of this clique, so it is avoided by a $((k-1)d)$-edge.	
\end{proof}

\begin{cor}
	$m_k(\F)> (k-1)\lfloor\frac{VC\text{--}dim(\HH)+1}k\rfloor$ for any $\HH\in \F$.\\
	In particular, $m_2(\F)> \lceil\frac{VC\text{--}dim(\HH)}2\rceil$.
\end{cor}
\begin{proof}
	Let $d=\lfloor\frac{VC\text{--}dim(\HH)+1}k\rfloor$ and apply Lemma \ref{lem:VC}.
\end{proof}

These restrictions rule out many natural constructions, but we believe that they should exist, and $m_2(\F)=3, m_3(\F)=\infty$ is possible.
However, we are not aware of any methods that could be applicable to show that each 3-heavy restricted subhypergraph of some \HH is 2-colorable.
Nevertheless, a clever construction might do the job.
Improving our bound $m_3(\F)=6$ a bit higher might not even require new ideas, but just a simple search through constructions similar to the one in the proof of Theorem \ref{thm:main}.

\bigskip

\textbf{Acknowledgment.}
I am eternally indebted to Balázs Keszegh for reading the first version of this manuscript, and for several useful comments.

%Just to be on the safe side, let us also mention three alternative, weaker conjectures:\\
%$m_k(\F)\le C\cdot k\cdot m_2(\F)$,\\
%$m_k(\F)\le C(m_2(\F))\cdot k$.

\end{document}